\documentclass{amsart}
\usepackage[utf8]{inputenc}
\usepackage{amsmath,amsthm,amssymb,amsfonts,stmaryrd,mathrsfs,dsfont}
\usepackage{color}

\theoremstyle{plain}
\newtheorem{theorem}{Theorem}[section]
\newtheorem{proposition}[theorem]{Proposition}
\newtheorem{corollary}[theorem]{Corollary}
\newtheorem{lemma}[theorem]{Lemma}

\theoremstyle{definition}

\newtheorem{example}[theorem]{Example}

\theoremstyle{remark}

\DeclareMathOperator{\reg}{reg}
\DeclareMathOperator{\ureg}{ureg}

\DeclareMathOperator{\Sym}{Sym}
\DeclareMathOperator{\id}{id}

\begin{document}

\title[Green's relations and unit-regularity for transformation semigroups]{Green's relations and unit-regularity for semigroup of transformations whose characters are bijective}

\author[Mosarof Sarkar]{\bfseries Mosarof Sarkar}
\address{Department of Mathematics, Central University of South Bihar, Gaya, Bihar, India}
\email{mosarofsarkar@cusb.ac.in}
\author[Shubh N. Singh]{\bfseries Shubh N. Singh}
\address{Department of Mathematics, Central University of South Bihar, Gaya, Bihar, India}
\email{shubh@cub.ac.in}



\begin{abstract}
Let $X$ be a nonempty set and $\mathcal{P}=\{X_i\colon i\in I\}$ be a partition of $X$. Denote by $T(X, \mathcal{P})$ the semigroup of all transformations of $X$ that preserve $\mathcal{P}$. In this paper, we study the semigroup $\mathcal{B}(X,\mathcal{P})$ of all transformations $f\in T(X, \mathcal{P})$ such that $\chi^{(f)}\in \Sym(I)$, where $\Sym(I)$ is the symmetric group on $I$ and $\chi^{(f)}\colon I \to I$ is the character (map) of $f$ defined by $i\chi^{(f)}=j$ whenever $X_if\subseteq X_j$. We describe unit-regular elements in $\mathcal{B}(X,\mathcal{P})$, and determine when $\mathcal{B}(X,\mathcal{P})$ is a unit-regular semigroup. We alternatively prove that $\mathcal{B}(X,\mathcal{P})$ is a regular semigroup. We describe Green's relations on $\mathcal{B}(X,\mathcal{P})$, and prove that $\mathcal{D} = \mathcal{J}$ on $\mathcal{B}(X,\mathcal{P})$ when $\mathcal{P}$ is finite. We also give a necessary and sufficient condition for $\mathcal{D} = \mathcal{J}$ on $\mathcal{B}(X,\mathcal{P})$. We end the paper with a conjecture.
\end{abstract}
\maketitle

\section{Introduction} 
Throughout this paper, let $X$ be a nonempty set, let $\mathcal{P} = \{X_i \colon i\in I\}$ be a partition of $X$, and let $E$ be the equivalence relation on $X$ corresponding to the partition $\mathcal{P}$. Denote by $T(X)$ (resp. $\Sym(X)$) the full transformation semigroup (resp. symmetric group) on $X$. The semigroup $T(X)$ and its subsemigroups play a vital role in semigroup theory, since every semigroup can be embedded in some $T(Z)$ (cf. \cite[Theorem 1.1.2]{howie95}). This famous result is analogous to Cayley's theorem for groups, which states that every group can be embedded in some $\Sym(X)$.

\vspace{0.05cm}
We say that any transformation $f\colon X \to X$ \emph{preserves} the partition $\mathcal{P}$ if for every $X_i \in \mathcal{P}$, there exists $X_j \in \mathcal{P}$ such that $X_i f \subseteq X_j$. In 1994, Pei \cite{pei-sf94} first studied the subsemigroup $T(X,\mathcal{P})$ of $T(X)$ consisting of all transformations that preserve the partition $\mathcal{P}$. Using symbols,
\begin{align*}
	T(X,\mathcal{P})=&\{f\in T(X)\colon (\forall X_i\in \mathcal{P})\;(\exists X_j\in \mathcal{P})\; X_if\subseteq X_j\}\\
	=&\{f\in T(X)\colon \forall x,y\in X, (x,y)\in E \Rightarrow (xf, yf)\in E\}.
\end{align*}
In that paper, Pei \cite[Theorem 2.8]{pei-sf94} proved that $T(X,\mathcal{P})$ is exactly the semigroup of all continuous selfmaps on $X$ with respect to the topology having $\mathcal{P}$ as a basis. The semigroup $T(X,\mathcal{P})$ and its subsemigroups have been extensively studied by a number of authors (see \cite{araujo-cps15, araujo-sf09, fernandes-bmmss12, fernandes-ca14, pei-sf05, pei-ca05, pei-ca07, pei-ac11, sarkar-ca21, sarkar-jaa22} for some references).


\vspace{0.05cm}
In what follows, the letter $I$ denotes the index set of the partition $\mathcal{P}$. Let $f\in T(X, \mathcal{P})$. The \emph{character (map)} of $f$ is the selfmap $\chi^{(f)}\colon I \to I$ defined by $i\chi^{(f)} = j$ whenever $X_i f \subseteq X_j$. For finite $X$, the character $\chi^{(f)}$ with the notation $\bar{f}$ has been studied by Ara\'{u}jo et al. \cite{araujo-cps15}, Dolinka and East \cite{east-ca16}, and Dolinka et al. \cite{east-bams16}. For arbitrary $X$, the character $\chi^{(f)}$ was first considered by Purisang and Rakbud \cite[p. 220]{rakbud-ckms16}. The character $\chi^{(f)}$ of $f$ has also been received attention (see \cite{rakbud-kmj18,sarkar-ca21,sarkar-sf22, sarkar-jaa22}).


\vspace{0.05cm}
Using the notion of character $\chi^{(f)}$, Purisang and Rakbud \cite{rakbud-ckms16} introduced the following subsemigroup of $T(X, \mathcal{P})$:
\[\mathcal{B}(X,\mathcal{P})=\{f\in T(X, \mathcal{P})\colon \chi^{(f)}\in \Sym(I)\}.\]
In that paper, the authors \cite[Theorem 3.5(1)]{rakbud-ckms16} proved that $\mathcal{B}(X,\mathcal{P})$ is a regular semigroup. The semigroup $\mathcal{B}(X,\mathcal{P})$ generalizes both $T(X)$ and $\Sym(X)$ in the sense that $\mathcal{B}(X,\mathcal{P})=T(X)$ if $|\mathcal{P}|= 1$, and $\mathcal{B}(X,\mathcal{P}) = \Sym(X)$ if $\mathcal{P}$ consists of singletons.
When $\mathcal{P}$ is finite, the authors \cite[Corollary 3.5]{sarkar-ca21} proved that $\mathcal{B}(X,\mathcal{P}) = \Sigma(X,\mathcal{P}) = T_{E^*}(X)$, where
\begin{align*}
	\Sigma(X,\mathcal{P}) &= \{f\in T(X, \mathcal{P})\colon Xf \cap X_i \neq \varnothing \; \forall X_i \in \mathcal{P}\},\\
	T_{E^*}(X) &= \{f\in T(X)\colon \forall x,y\in X,\; (x,y)\in E \iff (xf,yf)\in E\}.
\end{align*}
For arbitrary $X$, both the semigroups $\Sigma(X,\mathcal{P})$ and $T_{E^*}(X)$ have been studied (see \cite{araujo-cps15, araujo-sf09, sarkar-ca21, sarkar-sf22, sarkar-jaa22} and \cite{deng-bims16,deng-sf10,deng-sf12,sun-jaa13}, respectively). In particular, the authors \cite{araujo-cps15, araujo-sf09} considered the semigroup $\Sigma(X,\mathcal{P})$ for finite $X$ to calculate the rank of the finite semigroup $T(X,\mathcal{P})$.


\vspace{0.05cm}
This paper is motivated by various results on $T(X,\mathcal{P})$ and its two subsemigroups $\Sigma(X,\mathcal{P})$ and $T_{E^*}(X)$. The rest of the paper is organized as follows. In the next section, we define concepts, introduce notation, and recall some results needed in this paper. In Section 3, we give an alternative proof of Theorem 3.5(1) in \cite{rakbud-ckms16}, which ascertain that the semigroup $\mathcal{B}(X, \mathcal{P})$ is regular. We next describe unit-regular elements in $\mathcal{B}(X, \mathcal{P})$, and determine when $\mathcal{B}(X, \mathcal{P})$ is a unit-regular semigroup. In Section 4, we describe Green's relations on $\mathcal{B}(X, \mathcal{P})$. We also give a necessary and sufficient condition for $\mathcal{D} = \mathcal{J}$ on $\mathcal{B}(X, \mathcal{P})$.


\section{Preliminaries and Notation}   

Let $X$ be a nonempty set. The cardinality of $X$ is denoted by $|X|$, and the identity map on $X$ is denoted by $\id_X$. For any sets $A$ and $B$,  let $A\setminus B$ denote the set $\{x\in A \colon x\notin B\}$. A \emph{partition} of $X$ is a collection of pairwise disjoint nonempty subsets, called \emph{blocks}, whose union is $X$. A \emph{trivial} partition is a partition that has only singleton blocks or a single block. A partition is \emph{uniform} if all its blocks have the same cardinality. A \emph{transversal} of an equivalence relation $\rho$ on $X$ is a subset of $X$ that contains exactly one element from each $\rho$-class. The set of all positive integers is denoted by $\mathbb{N}$. For $n\in \mathbb{N}$, let $[n]$ denote the set $\{1,\ldots, n\}$.


\vspace{0.05cm}

We compose mappings from left to right and denote their composition by juxtaposition. Let $f\colon X \to Y$ be a mapping. We denote by $xf$ the image of an element $x\in X$ under $f$. For any $A\subseteq X$  (resp. $B\subseteq Y$), we denote by $Af$ (resp. $Bf^{-1}$) the set $\{af\colon a\in A\}$ (resp. $\{x\in X\colon  xf \in B\}$). Furthermore, if $B=\{b\}$, then we write $bf^{-1}$ instead of $\{b\} f^{-1}$. Let $D(f)$ denote the set $Y\setminus Xf$ and $\textnormal{d}(f)=|D(f)|$. The \emph{kernel} of $f$, denoted by $\ker(f)$, is an equivalence on $X$ defined by $\ker(f) = \{(a,b)\in X \times X \colon af = bf\}$. The symbol $\pi(f)$ denotes the partition of $X$ induced by $\ker(f)$, and the symbol $T_f$ denote any transversal of $\ker(f)$. Note that $|X\setminus T_f|$ is independent of the choice of transversal of $\ker(f)$ (cf. \cite[p. 1356]{higgins-rsm98}). Set $\textnormal{c}(f) = |X\setminus T_f|$. Let $g\colon X \to X$ be a mapping. For any nonempty subset $A$ of the domain of $g$, the \emph{restriction} of $g$ to $A$ is the mapping $g_{\upharpoonright A}\colon A\to X$ defined by $x(g_{\upharpoonright A})=xg$ for all $x\in A$. Moreover if $B$ is a subset of the codomain of $g$ such that $Ag \subseteq B$, then we reserve the same notation $g_{\upharpoonright A}$ for the mapping from $A$ to $B$ that assigns $xg$ to each $x\in A$.


\vspace{0.05cm}
Let $S$ be a semigroup and $a\in S$. We say that $a$ is \textit{regular} in $S$ if there exists $b\in S$ such that $aba = a$. The set of all regular elements in $S$ is denoted by $\reg(S)$. If $\reg(S)=S$, then $S$ is a \emph{regular semigroup}. In addition, let $S$ contains the identity. Then the set of all unit elements of $S$ is denoted by $U(S)$. We say that $a$ is \textit{unit-regular} in $S$ if there exists $u\in U(S)$ such that $aua = a$. The set of all unit-regular elements in $S$ is denoted by $\ureg(S)$. If $\ureg(S) = S$, then $S$ is a \emph{unit-regular semigroup}. Note that $U(T(X))=\Sym(X)$, and denote by $S(X,\mathcal{P})$ the set of all unit elements of $T(X,\mathcal{P})$. It is evident that $S(X,\mathcal{P}) = T(X,\mathcal{P}) \cap \Sym(X)$ and $U(\mathcal{B}(X,\mathcal{P})) = S(X,\mathcal{P})$.

\vspace{0.05cm}

Let $S$ be a semigroup and $a, b \in S$. The Green's relations $\mathcal{L},\mathcal{R},\mathcal{H},\mathcal{D}$, and $\mathcal{J}$ on $S$ defined as follows: $(a,b)\in \mathcal{L}$ if $S^1a=S^1b$, $(a,b)\in \mathcal{R}$ if $aS^1=bS^1$, $\mathcal{H} = \mathcal{L} \cap \mathcal{R}$, $\mathcal{D} = \mathcal{L} \circ \mathcal{R}$, and $(a,b)\in \mathcal{J}$ if $S^1aS^1=S^1bS^1$, where $S^1$ is the semigroup $S$ with an identity adjoined (if necessary). If $\mathcal{K}$ is any Green’s relation on $S$, then the equivalence class of $a$ with respect to $\mathcal{K}$ is denoted by $K_a$. Since $\mathcal{L},\mathcal{R}$, and $\mathcal{J}$ are defined in terms of ideals, which are partially ordered by inclusion, we have induced partial orders on the sets of the equivalence classes of $\mathcal{L},\mathcal{R}$, and $\mathcal{J}$ \cite[p. 47, (2.1.3)]{howie95}: $L_a\leq L_b$ if $S^1a \subseteq S^1b$, $R_a\leq R_b$ if $aS^1\subseteq bS^1$, and $J_a\leq J_b$ if $S^1aS^1\subseteq S^1bS^1$.

\vspace{0.07cm}

We refer the reader to \cite{howie95} for any undefined concepts, notation, and results of semigroup theory. We end this section by stating a list of preliminary results.

\begin{lemma}\cite[Lemma 3.1]{sarkar-jaa22}\label{transversal}
Let $f\colon X\to Y$ and $g\colon Y\to X$ be mappings. If $fgf = f$, then $X(fg)$ is a transversal of the equivalence relation $\ker(f)$.
\end{lemma}

\begin{lemma}\cite[Lemma 2.3]{rakbud-ckms16}\label{rakbud-lemma}
We have $\chi^{(fg)}=\chi^{(f)}\chi^{(g)}$ for all $f,g\in T(X,\mathcal{P})$.
\end{lemma}

\begin{lemma}\cite[Lemma 5.2]{sarkar-ca21}\label{fam-func}
Let $\mathcal{P} = \{X_i\colon i\in I\}$ be a partition of $X$ and $f \in T(X)$. Then $f\in T(X,\mathcal{P})$ if and only if there exists a unique family $B(f, I) := \{f_{\upharpoonright_{X_i}}\colon i \in I\}$ such that the codomain of each $f_{\upharpoonright_{X_i}}$ is a block of $\mathcal{P}$.
\end{lemma}

\begin{theorem}\cite[Theorem 5.8]{sarkar-ca21}\label{nece-SXP}
Let $\mathcal{P} = \{X_i\colon i\in I\}$ be a partition of $X$ and $f\in T(X,\mathcal{P})$. Then $f\in S(X,\mathcal{P})$ if and only if
\begin{enumerate}
	\item[\rm(i)] every mapping of $B(f, I)$ is bijective;
	\item[\rm(ii)] $\chi^{(f)}\in \Sym(I)$.
\end{enumerate}
\end{theorem}

\section{Unit-regularity for $\mathcal{B}(X,\mathcal{P})$} 

In this section, we describe unit-regular elements in $\mathcal{B}(X,\mathcal{P})$ and then give a
necessary and sufficient condition for $\mathcal{B}(X,\mathcal{P})$ to be unit-regular. We begin by giving an alternative proof of the following proposition which was first appeared in \cite[Theorem 3.5(1)]{rakbud-ckms16}.

\begin{proposition}
The semigroup $\mathcal{B}(X,\mathcal{P})$ is regular.
\end{proposition}

\begin{proof}[\textbf{Proof}]
Let $f\in \mathcal{B}(X,\mathcal{P})$, and write $(\chi^{(f)})^{-1}=\alpha$. Fix $x'\in xf^{-1}$ for each $x\in Xf$, and also fix $x_i\in X_i$ for each $i\in I$. Define $g\in T(X)$ as follows: Given $x\in X$, there exists $i\in I$ such that $x\in X_i$; we let
\begin{eqnarray*}
xg=
\begin{cases}
	x'         & \text{if $x\in X_i\cap Xf$;}\\
	x_{i\alpha} & \text{if $x\in X_i\setminus Xf$}.
\end{cases}
\end{eqnarray*}	

It is easy to see that $g\in T(X,\mathcal{P})$ and $\chi^{(g)}=\alpha$, so $g\in \mathcal{B}(X,\mathcal{P})$. To prove $fgf = f$, let $x\in X$ and write $xf = y$. Then $x\in X_i$ for some $i\in I$. Therefore, since $y' \in yf^{-1}$, we obtain $x(fgf)=(xf)gf=(yg)f = y'f=y=xf$, which yields $fgf=f$. Hence $f\in \reg(\mathcal{B}(X,\mathcal{P}))$ as required.

\end{proof}


The following theorem describes unit-regular elements in $\mathcal{B}(X,\mathcal{P})$.

\begin{theorem}\label{ureg-element}
Let $f\in \mathcal{B}(X,\mathcal{P})$. Then $f\in\ureg(\mathcal{B}(X,\mathcal{P}))$ if and only if $\textnormal{c}(f_{\upharpoonright_{X_i}})= \textnormal{d}(f_{\upharpoonright_{X_i}})$ for all $i\in I$.	
\end{theorem}

\begin{proof}[\textbf{Proof}]
Suppose that $f\in\ureg(\mathcal{B}(X,\mathcal{P}))$. Then there exists $g\in U(\mathcal{B}(X,\mathcal{P}))$ such that $fgf=f$. This gives $\chi^{(f)}\chi^{(g)}\chi^{(f)}=\chi^{(f)}$ by Lemma \ref{rakbud-lemma}. Therefore $\chi^{(g)}=(\chi^{(f)})^{-1}$, since $\chi^{(f)},\chi^{(g)}\in \Sym(I)$.
Write $\chi^{(f)}=\alpha$.

\vspace{0.5mm}
To prove the desired result, let $i\in I$ and write $i\alpha=j$. Then $j\chi^{(g)}=i$. Therefore, since $fgf = f$, it is easy to see that $f_{\upharpoonright_{X_i}}g_{\upharpoonright_{X_j}}f_{\upharpoonright_{X_i}}= f_{\upharpoonright_{X_i}}$, where $f_{\upharpoonright_{X_i}}\colon X_i\to X_j$ (resp. $g_{\upharpoonright_{X_j}}\colon X_j\to X_i$) is a map of $B(f,I)$ (resp. $B(g,I)$). It follows from Lemma \ref{transversal} that $T_{f_{\upharpoonright_{X_i}}} := X_i(f_{\upharpoonright_{X_i}}g_{\upharpoonright_{X_j}})$ is a transversal of $\ker(f_{\upharpoonright_{X_i}})$. Since $g\in S(X,\mathcal{P})$, it follows from Theorem \ref{nece-SXP} that every map of $B(g,I)$ is bijective. Therefore
\[(X_j\setminus X_if_{\upharpoonright_{X_i}})g_{\upharpoonright_{X_j}}= X_jg_{\upharpoonright_{X_j}}\setminus X_i(f_{\upharpoonright_{X_i}}g_{\upharpoonright_{X_j}}) = X_i\setminus T_{f_{\upharpoonright_{X_i}}},\]
which yields $|X_j\setminus X_if_{\upharpoonright_{X_i}}|=|X_i\setminus T_{f_{\upharpoonright_{X_i}}}|$. Hence $\textnormal{c}(f_{\upharpoonright_{X_i}})=|X_i\setminus T_{f_{\upharpoonright_{X_i}}}|=|X_j\setminus X_if_{\upharpoonright_{X_i}}|=\textnormal{d}(f_{\upharpoonright_{X_i}})$.

\vspace{1mm}	
Conversely, suppose that $\textnormal{c}(f_{\upharpoonright_{X_i}})= \textnormal{d}(f_{\upharpoonright_{X_i}})$ for all $i\in I$. To prove $f\in\ureg(\mathcal{B}(X,\mathcal{P}))$, we shall construct $g\in U(\mathcal{B}(X,\mathcal{P}))$ such that $fgf=f$. For this, let $i\in I$ and write $i(\chi^{(f)})^{-1}=j$. Let $T_{f_{\upharpoonright_{X_j}}}$ be a transversal of $\ker(f_{\upharpoonright_{X_j}})$, where $f_{\upharpoonright_{X_j}}\colon X_j\to X_i$ is a map of $B(f,I)$. It is easy to see that the map $\varphi_i\colon X_jf_{\upharpoonright_{X_j}}\to T_{f_{\upharpoonright_{X_j}}}$ defined by $x\varphi_i=x'$ whenever $x(f_{\upharpoonright_{X_j}})^{-1}\cap T_{f_{\upharpoonright_{X_j}}}=\{x'\}$ is bijective. Since $|X_i\setminus X_jf_{\upharpoonright_{X_j}}| = | X_j\setminus T_{f_{\upharpoonright_{X_j}}}|$ by hypothesis, there exists a bijection $\psi_i \colon X_i\setminus X_jf_{\upharpoonright_{X_j}}\to X_j\setminus T_{f_{\upharpoonright_{X_j}}}$. Define $h_i\colon X_i\to X_j$ by
\begin{eqnarray*}
xh_i=
\begin{cases}
	x\varphi_i  & \text{if $x\in X_jf_{\upharpoonright_{X_j}}$;}\\
	x\psi_i     & \text{if $x\in X_i\setminus X_jf_{\upharpoonright_{X_j}}$}.
\end{cases}
\end{eqnarray*}	
Clearly $h_i$ is bijective. We now prove that the map $g\in T(X)$ defined by $xg=xh_i$ whenever $x\in X_i$ for some $i\in I$, is a unit element of $\mathcal{B}(X,\mathcal{P})$. For this, let $i\in I$ and write $i(\chi^{(f)})^{-1}=j$. Then, since $h_i\colon X_i\to X_j$ is bijective, we get $X_ig=X_ih_i=X_j$. Therefore $g\in T(X,\mathcal{P})$ and $\chi^{(g)}=(\chi^{(f)})^{-1}$. Since $\chi^{(g)}\in \Sym(I)$ and every map $g_{\upharpoonright_{X_i}} = h_i$ of $B(g,I)$ is bijective, it follows from Theorem \ref{nece-SXP} that $g\in U(\mathcal{B}(X,\mathcal{P}))$. Finally, we prove that $fgf=f$. For this, let $x\in X$. Then $x\in X_i$ for some $i\in I$. Write $xf_{\upharpoonright_{X_i}}=y$ and $i\chi^{(f)}=j$. Therefore, since $y'\in y(f_{\upharpoonright_{X_i}})^{-1}\cap T_{f_{\upharpoonright_{X_i}}}$, we obtain $x(fgf)=(xf_{\upharpoonright_{X_i}})g_{\upharpoonright_{X_j}}f_{\upharpoonright_{X_i}}=(yg_{\upharpoonright_{X_j}})f_{\upharpoonright_{X_i}}=y'f_{\upharpoonright_{X_i}}=y=xf_{\upharpoonright_{X_i}}=xf$, which yields $fgf=f$. Hence $f\in\ureg(\mathcal{B}(X,\mathcal{P}))$ as required.
\end{proof}	

It is well-known that $T(X)$ is unit-regular if and only if $X$ is finite (cf. \cite[Proposition 5]{d'alarco-sf80}). In the following theorem, we describe the unit-regularity for $\mathcal{B}(X,\mathcal{P})$.

\begin{theorem}
The semigroup $\mathcal{B}(X,\mathcal{P})$ is unit-regular if and only if
\begin{enumerate}
\item[\rm(i)] $\mathcal{P}$ is a uniform partition;
\item[\rm(ii)] the cardinality of every block of $\mathcal{P}$ is finite.
\end{enumerate}
\end{theorem}

\begin{proof}[\textbf{Proof}]
Suppose that $\mathcal{B}(X,\mathcal{P})$ is a unit-regular semigroup.
\begin{enumerate}
\item[\rm(i)] Suppose to the contrary that there are distinct $j,k\in I$ such that $|X_j|\neq |X_k|$. Consider the following two possible cases:

\vspace{0.7mm}
\noindent Case 1: Suppose $|X_j|<|X_k|$. Then there exists a mapping $\varphi\colon X_j\to X_k$ that is injection, but not surjection. Therefore $\textnormal{c}(\varphi)=0$ and $\textnormal{d}(\varphi)\geq 1$.

\vspace{0.7mm}
\noindent Case 2: Suppose $|X_j|>|X_k|$. Then there exists a mapping $\varphi\colon X_j\to X_k$ that is surjection, but not injection. Therefore $\textnormal{c}(\varphi)\geq 1$ and $\textnormal{d}(\varphi)=0$.

\vspace{0.5mm}
In either case, there is a mapping $\varphi\colon X_j\to X_k$ such that $\textnormal{c}(\varphi)\neq \textnormal{d}(\varphi)$. We now choose $x_i\in X_i$ for each $i\in I\setminus \{k\}$, and consider $\alpha\in \Sym(I)$ such that $j\alpha=k$. Define $f\in T(X)$ by
\begin{eqnarray*}
xf=
\begin{cases}
x\varphi     & \text{if $x\in X_j$;}\\
x_{i\alpha}  & \text{if $x\in X_i$, where $i\in I\setminus \{j\}$}.
\end{cases}
\end{eqnarray*}
Clearly $f\in T(X,\mathcal{P})$ and $\chi^{(f)}=\alpha$, so $f\in \mathcal{B}(X,\mathcal{P})$. However, since $\textnormal{c}(\varphi)\neq \textnormal{d}(\varphi)$, we see that $\textnormal{c}(f_{\upharpoonright_{X_j}})=\textnormal{c}(\varphi)\neq \textnormal{d}(\varphi)=\textnormal{d}(f_{\upharpoonright_{X_j}})$. Therefore $f\notin \ureg(\mathcal{B}(X,\mathcal{P}))$ by Theorem \ref{ureg-element}, a contradiction. Hence $\mathcal{P}$ is uniform.

\vspace{0.8mm}
\item[\rm(ii)] Suppose to the contrary that $X_i$ is infinite for some $i\in I$. Then there exists a map $\psi \colon X_i\to X_i$ that is surjection, but not injection. Therefore $\textnormal{c}(\psi)\geq 1$ and $\textnormal{d}(\psi)=0$, so $\textnormal{c}(\psi)\neq \textnormal{d}(\psi)$. Define $g\in T(X)$ by
\begin{eqnarray*}
xg=
\begin{cases}
x\psi  & \text{if $x\in X_i$;}\\
x      & \text{otherwise.}
\end{cases}
\end{eqnarray*}
Clearly $g\in T(X,\mathcal{P})$ and
$\chi^{(g)}=\id_I$, so $g\in \mathcal{B}(X,\mathcal{P})$. However, since $\textnormal{c}(\psi)\neq \textnormal{d}(\psi)$, we see that $\textnormal{c}(f_{\upharpoonright_{X_i}})=\textnormal{c}(\psi)\neq \textnormal{d}(\psi)=\textnormal{d}(f_{\upharpoonright_{X_i}})$. Therefore $f\notin \ureg(\mathcal{B}(X,\mathcal{P}))$ by Theorem \ref{ureg-element}, a contradiction. Hence the cardinality of every block of $\mathcal{P}$ is finite.
\end{enumerate}

Conversely, suppose that the given conditions hold. Let $f\in \mathcal{B}(X,\mathcal{P})$ and let $i\in I$. Consider the map $f_{\upharpoonright_{X_i}}\colon X_i\to X_{i\chi^{(f)}}$ of $B(f,I)$. By condition (i), we have $|X_i|=|X_{i\chi^{(f)}}|$. Note from condition (ii) that $X_i$ is finite, so $\textnormal{c}(f_{\upharpoonright_{X_i}})= \textnormal{d}(f_{\upharpoonright_{X_i}})$. Thus $\textnormal{c}(f_{\upharpoonright_{X_i}})=\textnormal{d}(f_{\upharpoonright_{X_i}})$ for all $i\in I$, where $f_{\upharpoonright_{X_i}}\in B(f,I)$. Hence $f\in \ureg(\mathcal{B}(X,\mathcal{P}))$ by Theorem \ref{ureg-element} as required.
\end{proof}	


\section{Green's Relations on $\mathcal{B}(X,\mathcal{P})$} 

In this section, we give a complete description of Green's relations on $\mathcal{B}(X,\mathcal{P})$. We begin with a lemma that is useful in describing the relation $\mathcal{L}$ on $\mathcal{B}(X,\mathcal{P})$.

\begin{lemma}\label{b-l-lemma}
Let $f,g\in \mathcal{B}(X,\mathcal{P})$. Then $L_f\leq L_g$ in $\mathcal{B}(X,\mathcal{P})$ if and only if there exists $\alpha \in \Sym(I)$ such that $X_if\subseteq X_{i\alpha}g$ for all $i\in I$.	
\end{lemma}

\begin{proof}[\textbf{Proof}]
Suppose that $L_f\leq L_g$ in $\mathcal{B}(X,\mathcal{P})$. Then there exists $h\in \mathcal{B}(X,\mathcal{P})$ such that $f=hg$. This gives $\chi^{(f)}=\chi^{(h)}\chi^{(g)}$ by Lemma \ref{rakbud-lemma}. Write $\chi^{(h)}=\alpha$. Then, since $h\in \mathcal{B}(X,\mathcal{P})$, we have $\alpha\in \Sym(I)$. To prove the desired result, let $i\in I$. Note that $X_ih\subseteq X_{i\alpha}$, and therefore $X_if=X_i(hg)=(X_ih)g\subseteq X_{i\alpha}g$ as required.

\vspace{0.7mm}
Conversely, suppose that the given condition holds. To prove $L_f\leq L_g$ in $\mathcal{B}(X,\mathcal{P})$, we shall construct $h\in \mathcal{B}(X,\mathcal{P})$ such that $f=hg$. For this, let $i\in I$. By hypothesis, we have $X_if\subseteq X_{i\alpha}g$. Therefore for each $x\in X_i$, fix $x'\in X_{i\alpha}$ such that $xf=x'g$. Define $h\in T(X)$ as follows: Given $x\in X$, there exists $i\in I$ such that $x\in X_i$; we let $xh=x'$. Clearly $X_ih\subseteq X_{i\alpha}$, since $x' \in X_{i\alpha}$. Therefore $f\in T(X,\mathcal{P})$ and $\chi^{(h)}=\alpha$, so $h\in \mathcal{B}(X,\mathcal{P})$. Finally, we prove that $f=hg$. For this, let $x\in X$. Then $x\in X_i$ for some $i\in I$. Therefore, since $xf=x'g$, we obtain $x(hg)=(xh)g=x'g=xf$, which yields $f=hg$. Thus $L_f\leq L_g$ in $\mathcal{B}(X,\mathcal{P})$.
\end{proof}	


In the following theorem, we describe the relation $\mathcal{L}$ on $\mathcal{B}(X,\mathcal{P})$.

\begin{theorem}\label{b-l-green}
Let $f,g\in \mathcal{B}(X,\mathcal{P})$. Then $(f,g)\in \mathcal{L}$ in $\mathcal{B}(X,\mathcal{P})$ if and only if there exists $\alpha \in \Sym(I)$ such that $X_if=X_{i\alpha}g$ for all $i\in I$.	
\end{theorem}

\begin{proof}[\textbf{Proof}]

Suppose that $(f,g)\in \mathcal{L}$ in $\mathcal{B}(X,\mathcal{P})$. Then there exist $h,h'\in \mathcal{B}(X,\mathcal{P})$ such that $f=hg$ and $g=h'f$. It follows from Lemma \ref{rakbud-lemma} that $\chi^{(f)}=\chi^{(h)}\chi^{(g)}$ and $\chi^{(g)}=\chi^{(h')}\chi^{(f)}$. Note that $\chi^{(f)},\chi^{(g)},\chi^{(h)},\chi^{(h')}\in \Sym(I)$, write $\chi^{(h)}=\alpha$ and $\chi^{(h')}=\beta$. Clearly $\beta=\alpha^{-1}$. To prove the desired result, let $i\in I$. Note that $X_ih\subseteq X_{i\alpha}$, and so $X_if=X_i(hg)=(X_ih)g\subseteq X_{i\alpha}g$. For the reverse inclusion, note that $\beta=\alpha^{-1}$ and $X_{i\alpha}h'\subseteq X_{(i\alpha)\beta} = X_i$. Therefore, since $g= h'f$, we obtain $X_{i\alpha}g=X_{i\alpha}(h'f)=(X_{i\alpha}h')f\subseteq X_{i}f$ as required.

\vspace{0.7mm}
Conversely, suppose that there exists $\alpha \in \Sym(I)$ such that $X_if=X_{i\alpha}g$ for all $i\in I$. Then $L_f\leq L_g$ in $\mathcal{B}(X,\mathcal{P})$ by Lemma \ref{b-l-lemma}. Next, note that $\alpha^{-1}\in \Sym(I)$ and write $\alpha^{-1}=\beta$. Then by hypothesis, we get $X_ig = X_{(i\beta)\alpha}g= X_{i\beta}f$ for all $i\in I$. Therefore $L_g\leq L_f$ in  $\mathcal{B}(X,\mathcal{P})$ by Lemma \ref{b-l-lemma}. Thus $(f,g)\in \mathcal{L}$ in $\mathcal{B}(X,\mathcal{P})$.	
\end{proof}	

To describe the relation $\mathcal{R}$ on $\mathcal{B}(X,\mathcal{P})$, we need the following terminology and Lemma \ref{b-r-lemma}.

\vspace{1mm}

Let $f,g\in T(X)$. We say that $\pi(f)$ \emph{refines} $\pi(g)$, denoted by $\pi(f)\preceq \pi(g)$, if $\ker(f)\subseteq \ker(g)$. Equivalently, we have $\pi(f)\preceq \pi(g)$ if for every $P\in \pi(f)$, there exists $Q\in \pi(g)$ such that $P\subseteq Q$. We shall write $\pi(f)=\pi(g)$ if $\pi(f)\preceq \pi(g)$ and $\pi(g)\preceq \pi(f)$.

\vspace{1mm}

\begin{lemma}\label{b-r-lemma}
Let $f,g\in \mathcal{B}(X,\mathcal{P})$. Then $R_f\leq R_g$ in $\mathcal{B}(X,\mathcal{P})$ if and only if $\pi(g)\preceq \pi(f)$.	
\end{lemma}

\begin{proof}[\textbf{Proof}]
Suppose that $R_f\leq R_g$ in $\mathcal{B}(X,\mathcal{P})$. Then there exists $h\in \mathcal{B}(X,\mathcal{P})$ such that $f=gh$. It follows from \cite[Lemma 2.6]{clifford-ams61} that $\pi(g)\preceq \pi(f)$, since $f,g,h\in T(X)$.

\vspace{0.7mm}
Conversely, suppose that $\pi(g)\preceq \pi(f)$. To prove $R_f\leq R_g$ in $\mathcal{B}(X,\mathcal{P})$, we shall construct $h\in \mathcal{B}(X,\mathcal{P})$ such that $f=gh$. For this, fix $y_i\in X_i$ for each $i\in I$; fix $x'\in xg^{-1}$ for each $x\in Xg$. Note that $\chi^{(f)}, \chi^{(g)} \in \Sym(I)$, and write $(\chi^{(g)})^{-1}\chi^{(f)}=\alpha$. Define $h\in T(X)$ as follows: Given $x\in X$, there exists $i\in I$ such that $x\in X_i$; we let
\begin{eqnarray*}
xh=
\begin{cases}
x'f         & \text{if $x\in X_i\cap Xg$;}\\
y_{i\alpha} & \text{if $x\in X_i\setminus Xg$}.
\end{cases}
\end{eqnarray*}
Clearly $h$ is well-defined, since $\pi(g)\preceq \pi(f)$. We also observe that $X_ih\subseteq X_{i\alpha}$, since $x' \in X_{i(\chi^{(g)})^{-1}}$ whenever $x\in X_i \cap Xg$, and $(\chi^{(g)})^{-1}\chi^{(f)}=\alpha$. Therefore $h\in T(X,\mathcal{P})$ and $\chi^{(h)}=\alpha$, so $h\in \mathcal{B}(X,\mathcal{P})$. Finally, we prove that $f=gh$. For this, let $x\in X$. Then $x\in X_i$ for some $i\in I$. Therefore, since $x'\in xg^{-1}$ and $xg^{-1}\in \pi(f)$, we obtain $x(gh)=(xg)h = x'f=xf$, which yields $f=gh$. Thus $R_f\leq R_g$ in $\mathcal{B}(X,\mathcal{P})$.
\end{proof}	
\begin{theorem}\label{b-r-green}
Let $f,g\in \mathcal{B}(X,\mathcal{P})$. Then $(f,g)\in \mathcal{R}$ in $\mathcal{B}(X,\mathcal{P})$ if and only if $\pi(f)=\pi(g)$.	
\end{theorem}

\begin{proof}[\textbf{Proof}]
This proof follows directly from Lemma \ref{b-r-lemma}.	
\end{proof}	

To describe the relation $\mathcal{D}$ on $\mathcal{B}(X,\mathcal{P})$, we need the following terminology.

\vspace{1mm}

For any $A\subseteq X$ and any $f\in T(X)$, let $\pi_A(f)=\{M \in \pi(f)\colon M\cap A\neq \varnothing\}$. We shall write $\pi(f)$ instead of $\pi_X(f)$ if $A=X$.

\vspace{1mm}

\begin{theorem}\label{b-d-green}
Let $f,g\in \mathcal{B}(X,\mathcal{P})$. Then $(f,g)\in \mathcal{D}$ in $\mathcal{B}(X,\mathcal{P})$ if and only if there exists $\alpha \in \Sym(I)$ such that $|X_if|=|X_{i\alpha}g|$ for all $i\in I$.	
\end{theorem}

\begin{proof}[\textbf{Proof}]
Suppose that $(f,g)\in \mathcal{D}$ in $\mathcal{B}(X,\mathcal{P})$. Then there exists $h\in \mathcal{B}(X,\mathcal{P})$ such that $(f,h)\in \mathcal{L}$ in $\mathcal{B}(X,\mathcal{P})$ and $(h,g)\in \mathcal{R}$ in $\mathcal{B}(X,\mathcal{P})$. Since $(f,h)\in \mathcal{L}$, it follows from Theorem \ref{b-l-green} that there exists $\alpha\in \Sym(I)$ such that $X_kf=X_{k\alpha}h$ for all $k\in I$. To prove the desired result, let $i\in I$ and write $i\alpha=j$. Then $X_if=X_jh$, so $|X_if|=|X_jh|$. Now since $(h,g)\in \mathcal{R}$ in $\mathcal{B}(X,\mathcal{P})$, it follows from Theorem \ref{b-r-green} that $\pi(h)=\pi(g)$. This yields $\pi_{X_j}(h)=\pi_{X_j}(g)$, so $|\pi_{X_j}(h)|=|\pi_{X_j}(g)|$. Notice that $|\pi_{X_j}(h)|=|X_jh|$ and $|\pi_{X_j}(g)|=|X_jg|$, so $|X_jh|=|X_jg|$. Thus $|X_if|=|X_jg|$ as required.

\vspace{0.5mm}
Conversely, suppose that there exists $\alpha \in \Sym(I)$ such that $|X_if|=|X_{i\alpha}g|$ for all $i\in I$. To prove $(f,g)\in \mathcal{D}$ in $\mathcal{B}(X,\mathcal{P})$, we shall construct $h\in \mathcal{B}(X,\mathcal{P})$ such that $(f,h)\in \mathcal{L}$ in $\mathcal{B}(X,\mathcal{P})$ and $(h,g)\in \mathcal{R}$ in $\mathcal{B}(X,\mathcal{P})$. For this, let $i\in I$ and write $\beta=\alpha^{-1}$. Then $|X_ig| = |X_{(i\beta)\alpha}g| = |X_{i\beta}f|$ by hypothesis, so there is a bijection $\varphi_i\colon X_ig\to X_{i\beta}f$. Now define $h\in T(X)$ as follows: Given $x\in X$, there exists $i\in I$ such that $x\in X_i$; we let $xh=x(g\varphi_i)$. Clearly $h\in T(X, \mathcal{P})$ and $\chi^{(h)} = \beta \chi^{(f)}$, since $X_ih=X_i(g\varphi_i)=X_{i\beta}f \subseteq X_{i(\beta\chi^{(f)})}$ for all $i\in I$. Therefore $h\in \mathcal{B}(X,
\mathcal{P})$. Recall that $\beta\in \Sym(I)$ and $X_ih= X_{i\beta}f$ for all $i\in I$. Thus $(f,h)\in \mathcal{L}$ in $\mathcal{B}(X,
\mathcal{P})$ by Theorem \ref{b-l-green} and the fact that $\mathcal{L}$ is a symmetric relation. It is easy to see from definition of $h$ that $\pi_{X_i}(h)=\pi_{X_i}(g)$ for all $i\in I$, which yields $\pi(h)=\pi(g)$. Therefore $(h,g)\in \mathcal{R}$ in $\mathcal{B}(X,\mathcal{P})$ by Theorem \ref{b-r-green}. Thus, since $\mathcal{D} = \mathcal{L} \circ \mathcal{R}$, we conclude that $(f,g)\in \mathcal{D}$ in $\mathcal{B}(X,\mathcal{P})$.
\end{proof}

To describe the relation $\mathcal{J}$ on $\mathcal{B}(X,\mathcal{P})$, we need the following lemma.

\begin{lemma}\label{b-j-lemma}
Let $f,g\in \mathcal{B}(X,\mathcal{P})$. Then $J_f\leq J_g$ in $\mathcal{B}(X,\mathcal{P})$ if and only if there exists $\alpha \in \Sym(I)$ such that $|X_if|\leq |X_{i\alpha}g|$ for all $i\in I$.
\end{lemma}

\begin{proof}[\textbf{Proof}]
Suppose that $J_f\leq J_g$ in $\mathcal{B}(X,\mathcal{P})$. Then there exist $h,h_1\in \mathcal{B}(X,\mathcal{P})$ such that $f=hgh_1$. Set $\chi^{(h)}=\alpha$. Clearly $\alpha\in \Sym(I)$, since $h\in \mathcal{B}(X,\mathcal{P})$. To prove the desired result, let $i\in I$. Notice that $X_ih\subseteq X_{i\alpha}$ and therefore
$|X_if|=|(X_ih)gh_1|\leq |(X_{i\alpha}g)h_1|\leq |X_{i\alpha}g|$ as required.

\vspace{0.7mm}
Conversely, suppose that there exists $\alpha \in \Sym(I)$ such that $|X_if|\leq |X_{i\alpha}g|$ for all $i\in I$. Then for every $i\in I$, there exists an injection $\varphi_i\colon X_if\to X_{i\alpha}g$. Now choose $y'\in (y\varphi_i)g^{-1}$ for each $y\in X_if$. To prove the desired result, we shall construct $h,h_1\in \mathcal{B}(X,\mathcal{P})$ such that $f=hgh_1$.

\vspace{0.05cm}
We first define $h\in T(X)$ as follows: Given $x\in X$, there exists $i\in I$ such that $x\in X_i$; write $xf = y$, and we let $xh = y'$. To prove $h\in \mathcal{B}(X,\mathcal{P})$, let $i\in I$ and $x\in X_i$. Set $xf = z$. Then $xh = z'$ by definition of $h$, where $z' \in (z\varphi_i)g^{-1}$. Since $z\in X_i f$, it follows that $z\varphi_i\in X_{i\alpha}g$ by definition of $\varphi_i$. This gives
$(z\varphi_i)g^{-1}\subseteq X_{i\alpha}$, since $\chi^{(g)}\in \Sym(I)$.
Therefore, since $z' \in (z\varphi_i)g^{-1}$, we get $xh = z' \in X_{i\alpha}$. Thus $h\in T(X,\mathcal{P})$ and $\chi^{(h)}=\alpha$, so $h\in \mathcal{B}(X,\mathcal{P})$.

\vspace{0.05cm}
To define $h_1\in \mathcal{B}(X,\mathcal{P})$, we observe for every $i\in I$ that $X_i(f\varphi_i)=X_i(hg)$, since $(xf)' \in ((xf)\varphi_i)g^{-1}$ and $x(hg)=(xh)g = (xf)'g=(xf)\varphi_i$ for all $x\in X_i$.  Also, fix $x_i\in X_i$ for each $i\in I$. Let $\beta=(\alpha\chi^{(g)})^{-1}\chi^{(f)}$. Clearly $\beta\in \Sym(I)$ and $\chi^{(f)}=\alpha\chi^{(g)}\beta$, since  $\chi^{(f)},\chi^{(g)},\alpha\in \Sym(I)$. Define $h_1\in T(X)$ as follows: Given $x\in X$, there exists $i\in I$ such that $x\in X_i$; write $i(\alpha\chi^{(g)})^{-1}=j$, and we let
\begin{eqnarray*}
xh_1=
\begin{cases}
	y          & \text{if $x\in X_i\cap X(hg)$ and $x\varphi_j^{-1}=\{y\}$;}\\
	x_{i\beta} & \text{if $x\in X_i\setminus X(hg)$}.	
\end{cases}
\end{eqnarray*}
To prove $h_1\in \mathcal{B}(X,\mathcal{P})$, let $i\in I$ and $x\in X_i$. Consider the following two possible cases:

\vspace{0.5mm}
\noindent Case 1: Suppose $x\in X_i\cap X(hg)$. Write $i(\alpha\chi^{(g)})^{-1}=k$.
Then, since $ \chi^{(h)}= \alpha$, we see that $X_k (hg) \subseteq X_{k\alpha}g \subseteq X_{k(\alpha\chi^{(g)})} = X_i$. As $\alpha, \chi^{(g)} \in \Sym(I)$, it follows that $X_k(hg) = X_i \cap X (hg)$. Therefore $x\in X_k(hg)\subseteq X_{k\alpha}g$. Note that $\varphi_k\colon X_kf\to X_{k\alpha}g$ is injective and $X_k(f\varphi_k)=X_k(hg)$. It follows that $x\varphi_k^{-1}=\{z\}$ for some $z\in X_kf$. Therefore, since $i(\alpha\chi^{(g)})^{-1}=k$ and $\beta=(\alpha\chi^{(g)})^{-1}\chi^{(f)}$, we obtain $k\chi^{(f)}=(i(\alpha\chi^{(g)})^{-1})\chi^{(f)}=i\beta$, which yields $X_kf\subseteq X_{i\beta}$. Thus, since $z\in X_kf$, we get $xh_1 = z\in X_{i\beta}$.

\vspace{0.5mm}
\noindent Case 2: Suppose $x\in X_i\setminus X(hg)$. Then $xh_1=x_{i\beta}$ by definition of $h_1$, so $xh_1\in X_{i\beta}$.

\vspace{0.1cm}
In either case, we have $xh_1\in X_{i\beta}$. Hence $X_ih_1\subseteq X_{i\beta}$. This gives $h_1\in T(X,\mathcal{P})$ and $\chi^{(h_1)}=\beta$, so $h_1\in \mathcal{B}(X,\mathcal{P})$.

\vspace{0.1cm}
Finally, we prove that $f=hgh_1$. For this, let $x\in X$. Then $x\in X_i$ for some $i\in I$. Recall that $\varphi_i\colon X_if\to X_{i\alpha}g$ is an injection. Therefore, since $xh=(xf)'\in ((xf)\varphi_i)g^{-1}\subseteq X_{i\alpha}$, we obtain
$x(hgh_1)=(xh)(gh_1)=(xf)'(gh_1)=(x(f\varphi_i))h_1=xf$,
which yields $f=hgh_1$ as required.
\end{proof}
\vspace{0.5mm}
\begin{theorem}\label{b-j-green}
Let $f,g\in \mathcal{B}(X,\mathcal{P})$. Then $(f,g)\in \mathcal{J}$ in $\mathcal{B}(X,\mathcal{P})$ if and only if there exist $\alpha,\beta\in \Sym(I)$ such that $|X_if|\leq |X_{i\alpha}g|$ and $|X_ig|\leq |X_{i\beta}f|$ for all $i\in I$.	
\end{theorem}

\begin{proof}[\textbf{Proof}]
This proof follows directly from Lemma \ref{b-j-lemma}.		
\end{proof}


\vspace{1mm}
In the following proposition, we prove that the relations $\mathcal{D}$ and $\mathcal{J}$ on $\mathcal{B}(X,\mathcal{P})$ coincide when $\mathcal{P}$ is finite.

\begin{proposition}\label{p-finite}
If $\mathcal{P}$ is a finite partition of $X$, then $\mathcal{D}=\mathcal{J}$ on $\mathcal{B}(X,\mathcal{P})$.
\end{proposition}
\begin{proof}[\textbf{Proof}]
Let $\mathcal{P}=\{X_i\colon i\in [m]\}$, where $m\in \mathbb{N}$. In general, we have $\mathcal{D}\subseteq \mathcal{J}$. For the reverse inclusion, let $(f,g)\in \mathcal{J}$ in $\mathcal{B}(X,\mathcal{P})$. Then by Theorem \ref{b-j-green}, there exist $\alpha,\beta\in \Sym([m])$ such that $|X_if|\leq |X_{i\alpha}g|$ and $|X_ig|\leq |X_{i\beta}f|$ for all $i\in [m]$. Therefore for every $i\in [m]$, we obtain
\begin{eqnarray*}
\qquad\qquad |X_if|\leq |X_{i\alpha}g|\leq|X_{(i\alpha)\beta}f|=|X_{i(\alpha\beta)}f| \qquad\qquad \cdots \text{(1).}
\end{eqnarray*}

Note that $m$ is finite and $\alpha\beta\in \Sym([m])$. Therefore there exists $k\in \mathbb{N}$ such that $j(\alpha\beta)^k=j$ for all $j\in [m]$. 	
\vspace{0.5mm}

In view of Theorem \ref{b-d-green}, in order to show that $(f,g)\in \mathcal{D}$ in $\mathcal{B}(X,\mathcal{P})$, it suffices to prove that $|X_if|=|X_{i\alpha}g|$ for all $i\in [m]$. For this, let $i\in [m]$. Then from inequality \rm(1), we have $|X_if|\leq |X_{i(\alpha\beta)}f|$. Therefore, since $(\alpha\beta)^k$ is identity map on $[m]$, we obtain by using inequality \rm(1) that
\[|X_if|\leq |X_{i(\alpha\beta)}f|\leq |X_{i(\alpha\beta)^2}f|\leq\cdots \leq |X_{i(\alpha\beta)^k}f|=|X_if|\]
whence $|X_if|=|X_{i(\alpha\beta)}f|$. Note from inequality \rm(1) that $|X_if|\leq |X_{i\alpha}g|\leq |X_{i(\alpha\beta)}f|$, so $|X_if|=|X_{i\alpha}g|$. Hence, since $\alpha\in \Sym([m])$, we conclude from  Theorem \ref{b-d-green} that $(f,g)\in \mathcal{D}$ in $\mathcal{B}(X,\mathcal{P})$ as required.
\end{proof}

We now introduce the following terminology.

\vspace{1.5mm}
For any $f\in \mathcal{B}(X,\mathcal{P})$, let $\mathcal{P}_f=\{X_if\colon i\in I\}$.
It is clear that $\mathcal{P}_f$ is a partition of the range set $Xf$ of $f$, since $\chi^{(f)}\in \Sym(I)$.

\vspace{1.5mm}

Note that if $\mathcal{P}$ is an infinite partition consisting of only singleton blocks, then $\mathcal{B}(X,\mathcal{P})=\Sym(X)$, so  $\mathcal{D}=\mathcal{J}$ on $\mathcal{B}(X,\mathcal{P})$. However, for an infinite partition $\mathcal{P}$, it is not always true that $\mathcal{D}=\mathcal{J}$ on $\mathcal{B}(X,\mathcal{P})$, as shown in the following example.

\begin{example}
Let $X=\mathbb{N}$. Consider the partition $\mathcal{P} = \{X_i \colon i\in I\}$ of $X$, where $I=\{1,3,5k,5k+1 \colon k\in \mathbb{N}\}$; $X_1=\{1,2\}$, $X_3=\{3,4\}$, and for each $k\in \mathbb{N}$, $X_{5k}=\{5k\}$, $X_{5k+1}=\{5k+1,5k+2,5k+3,5k+4\}$. Take mappings $f,g\colon X \to X$ defined by $xg = x$ and
\begin{eqnarray*}
xf=
\begin{cases}
1    &  \text{if $x\in X_1$;}\\	
x    &  \text{otherwise}.
\end{cases}
\end{eqnarray*}
Clearly $f,g\in T(X,\mathcal{P})$ and $\chi^{(f)}=\chi^{(g)}=\id_I$, so $f,g\in \mathcal{B}(X,\mathcal{P})$. Notice that $\mathcal{P}_f$ has exactly one block of cardinality $2$, while $\mathcal{P}_g$ has two blocks of cardinality $2$. Therefore there does not exist any $\alpha\in \Sym(I)$ such that $|X_if|=|X_{i\alpha}g|$ for all $i\in I$. Hence $(f,g)\notin \mathcal{D}$ in $\mathcal{B}(X,\mathcal{P})$ by Theorem \ref{b-d-green}.

\vspace{0.5mm}
Finally, we prove that $(f,g)\in \mathcal{J}$ in $\mathcal{B}(X,\mathcal{P})$. Consider $\alpha,\beta\in T(I)$ defined by $i\alpha=i$ and
\begin{eqnarray*}
i\beta=
\begin{cases}
6           & \text{if $i=1$;}\\
3           & \text{if $i=3$;}\\
1           & \text{if $i=5$;}\\
5(k-1)      & \text{if $i=5k$, where $k\geq 2$;}\\
5(k+1)+1    & \text{if $i=5k+1$, where $k\geq 1$}.
\end{cases}
\end{eqnarray*}
Observe that $\alpha,\beta\in \Sym(I)$. For every $i\in I$, it is also routine to verify that $|X_if|\leq |X_{i\alpha}g|$ and $|X_ig|\leq |X_{i\beta}f|$. Therefore $(f,g)\in \mathcal{J}$ in $\mathcal{B}(X,\mathcal{P})$ by Theorem \ref{b-j-green}, and thus $(f,g)\in \mathcal{J}\setminus \mathcal{D}$ in $\mathcal{B}(X,\mathcal{P})$.
\end{example}

In the following proposition, we give a sufficient condition under which the relations $\mathcal{D}$ and $\mathcal{J}$ on $\mathcal{B}(X,\mathcal{P})$ do not coincide.
\begin{proposition}\label{p3-infinite}
Let $\mathcal{P} = \{X_i \colon i\in I\}$ be a partition of $X$. If $J=\{i\in I \colon |X_i|\geq 3\}$ is an infinite subset of $I$, then $\mathcal{D}\neq\mathcal{J}$ on $\mathcal{B}(X,\mathcal{P})$.
\end{proposition}

\begin{proof}[\textbf{Proof}]
Since $J$ is infinite, there exists a proper subset $K$ of $J$ such that $|K|=|J\setminus K|=|J|$. Let $k\in K$, and fix distinct $y_k,y_k'\in X_k$. Also, fix distinct $u_i,v_i,w_i\in X_i$ for each $i\in K\setminus \{k\}$; fix $z_i\in X_i$ for each $i\in I\setminus K$. Consider mappings $f,g\colon X \to X$ defined by
\begin{eqnarray*}
xf=
\begin{cases}
y_k   &  \text{if $x=y_k$;}\\	
y_k'   &  \text{if $x\in X_k\setminus \{y_k\}$;}\\
u_i   &  \text{if $x=u_i\in X_i$, where $i\in K\setminus \{k\}$;}\\
v_i   &  \text{if $x=v_i\in X_i$, where $i\in K\setminus \{k\}$;}\\
w_i   &  \text{if $x\in X_i\setminus \{u_i,v_i\}$, where $i\in K\setminus \{k\}$;}\\
z_i   &  \text{if $x\in X_i$, where $i\in I\setminus K$}
\end{cases}
\end{eqnarray*}
and
\begin{eqnarray*}
xg=
\begin{cases}	
y_k   &  \text{if $x\in X_k$;}\\
u_i   &  \text{if $x=u_i\in X_i$, where $i\in K\setminus \{k\}$;}\\
v_i   &  \text{if $x=v_i\in X_i$, where $i\in K\setminus \{k\}$;}\\
w_i   &  \text{if $x\in X_i\setminus \{u_i,v_i\}$, where $i\in K\setminus \{k\}$;}\\
z_i   &  \text{if $x\in X_i$, where $i\in I\setminus K$}.
\end{cases}
\end{eqnarray*}
For every $i\in I$, we observe that $X_if\subseteq X_i$ and $X_ig\subseteq X_i$. Therefore $f,g\in T(X,\mathcal{P})$ and $\chi^{(f)}=\id_I= \chi^{(g)}$. It follows that $f,g\in \mathcal{B}(X,\mathcal{P})$. Notice that $|X_kf|=2$, while $|X_ig|\neq 2$ for all $i\in I$. Therefore there does not exist $\alpha\in \Sym(I)$ such that $|X_kf|=|X_{k\alpha}g|$. Hence $(f,g)\notin \mathcal{D}$ in $\mathcal{B}(X,\mathcal{P})$ by Theorem \ref{b-d-green}.

\vspace{0.5mm}
Finally, we prove that $(f,g)\in \mathcal{J}$ in $\mathcal{B}(X,\mathcal{P})$. For this, we first fix $k'(\neq k)\in K$. Since $K$ is infinite, we see that $|K\setminus \{k\}|=|K\setminus \{k,k'\}|$. Therefore there is a bijection $\varphi \colon K\setminus \{k\}\to K\setminus \{k,k'\}$. Recall that $J$ is an infinite subset of $I$. Since $|J\setminus K|=|J|$ and $J\setminus K\subseteq I\setminus K$, it follows that $I\setminus K$ is infinite. Therefore $|I\setminus K|=|(I\setminus K)\cup \{k\}|$, so there is a bijection $\psi \colon I\setminus K\to (I\setminus K)\cup \{k\}$. Consider mappings $\alpha,\beta\in T(I)$ defined by $i\beta=i$ and
\begin{eqnarray*}
i\alpha=
\begin{cases}
k'          & \text{if $i=k$;}\\
i\varphi    & \text{if $i\in K\setminus \{k\}$;}\\
i\psi       & \text{$i\in I\setminus K$}.
\end{cases}
\end{eqnarray*}
Notice that $\alpha,\beta\in \Sym(I)$. For every $i\in I$, it is routine to verify that $|X_if|\leq |X_{i\alpha}g|$ and $|X_ig|\leq |X_{i\beta}f|$. Therefore $(f,g)\in \mathcal{J}$ in $\mathcal{B}(X,\mathcal{P})$ by Theorem \ref{b-j-green}. Thus $\mathcal{D}\neq\mathcal{J}$ on $\mathcal{B}(X,\mathcal{P})$.
\end{proof}

Now we establish an alternative characterization of the relation $\mathcal{D}$ on $\mathcal{B}(X,\mathcal{P})$ in Proposition \ref{d-green-alter}. Before that, we introduce the following terminology.

\vspace{1mm}
Let $f\in \mathcal{B}(X,\mathcal{P})$. Recall that $\mathcal{P}_f=\{X_if\colon i\in I\}$ is a partition of the range set $Xf$ of $f$. For any cardinal $\lambda$, let $I_{\lambda}^f=\{i\in I\colon |X_if|=\lambda\}$ and $n_{\lambda}^f=|I_{\lambda}^f|$. Then it is clear that the collection $\{I_{\lambda}^f\colon \lambda\leq |X| ~\text{such that}~ I_{\lambda}^f\neq \varnothing\}$ is a partition of $I$.

\begin{proposition}\label{d-green-alter}
Let $f,g\in \mathcal{B}(X,\mathcal{P})$. Then $(f,g)\in \mathcal{D}$ in $\mathcal{B}(X,\mathcal{P})$ if and only if $n_{\lambda}^f=n_{\lambda}^g$ for all cardinals $\lambda$.
\end{proposition}
\begin{proof}[\textbf{Proof}]
Suppose that $(f,g)\in \mathcal{D}$ in $\mathcal{B}(X,\mathcal{P})$. Then by Theorem \ref{b-d-green}, there exists $\alpha\in \Sym(I)$ such that $|X_if|=|X_{i\alpha}g|$ for all $i\in I$. To prove the desired result, let $\lambda$ be a cardinal. It is clear that $n_{\lambda}^f=0$ if and only if $n_{\lambda}^g= 0$, since $|X_if|\neq \lambda$ for all $i\in I$ if and only if $|X_ig|\neq \lambda$ for all $i\in I$. Assume that $n_{\lambda}^f \ge 1$, and let $i\in I_{\lambda}^f$. Then $|X_{i\alpha}g|=|X_if|=\lambda$, so $i\alpha\in I_{\lambda}^g$. Since $\alpha$ is injective, we thus get $n_{\lambda}^f\leq n_{\lambda}^g$. Now, let $i\in I_{\lambda}^g$. Then $|X_{i\alpha^{-1}}f|=|X_{(i\alpha^{-1})\alpha}g|=|X_ig| =\lambda$, so $i\alpha^{-1}\in I_{\lambda}^f$. Since $\alpha^{-1}$ is injective, we thus get $n_{\lambda}^f\geq n_{\lambda}^g$. Hence $n_{\lambda}^f=n_{\lambda}^g$ as required.

\vspace{0.7mm}	
Conversely, suppose that the given condition holds. In view of Theorem \ref{b-d-green}, in order to prove $(f,g)\in \mathcal{D}$ in $\mathcal{B}(X,\mathcal{P})$, it suffices to construct $\alpha\in \Sym(I)$ such that $|X_if|=|X_{i\alpha}g|$ for all $i\in I$. For this, consider a cardinal $\lambda$. Note by hypothesis that $I_{\lambda}^f=\varnothing$ if and only if $I_{\lambda}^g=\varnothing$. If $I_{\lambda}^f\neq \varnothing$, then $|I_{\lambda}^f|=|I_{\lambda}^g|$  by hypothesis. Therefore there exists a bijection $\alpha_{\lambda}\colon I_{\lambda}^f\to I_{\lambda}^g$. Define $\alpha\in T(I)$ by $i\alpha=i\alpha_{\lambda}$ whenever $i\in I_{\lambda}^f$. Clearly $\alpha$ is well-defined, since $\lambda$ is arbitrary. Also, since every $\alpha_{\lambda}$ is bijective and $\displaystyle\bigcup_{\lambda} I_{\lambda}^f=I=\displaystyle\bigcup_{\lambda} I_{\lambda}^g$, we see that $\alpha\in \Sym(I)$. It is also routine to verify that $|X_if|=|X_{i\alpha}g|$ for all $i\in I$. Hence $(f,g)\in \mathcal{D}$ by Theorem \ref{b-d-green}.	
\end{proof}

In the following theorem, we give a sufficient condition for $D_f=J_f$ in $\mathcal{B}(X,\mathcal{P})$.

\begin{theorem}\label{two-consecutative}
Let $f\in \mathcal{B}(X,\mathcal{P})$. If there exists two consecutive cardinals $\lambda_1,\lambda_2$ such that $n_{\lambda}^f=0$ for all cardinals $\lambda\notin \{\lambda_1,\lambda_2\}$, then $D_f=J_f$ in $\mathcal{B}(X,\mathcal{P})$.	
\end{theorem}
\begin{proof}[\textbf{Proof}]
Clearly $D_f\subseteq J_f$, since $\mathcal{D}\subseteq \mathcal{J}$. For the reverse inclusion, let $g\in J_f$. Then by Theorem \ref{b-l-green}, there exist $\alpha,\beta\in \Sym(I)$ such that for all $i\in I$,
\begin{eqnarray*}
|X_if|\leq |X_{i\alpha}g|\hspace{1cm} \cdots \text{(i)} \qquad\text{and}\qquad |X_ig|\leq |X_{i\beta}f| \hspace{1cm} \cdots \text{(ii)}
\end{eqnarray*}

First, we claim that $n_{\lambda}^g=0$ for all cardinals $\lambda\notin \{\lambda_1,\lambda_2\}$. Suppose to the contrary that there exists $\lambda\notin \{\lambda_1,\lambda_2\}$ such that $n_{\lambda}^g\neq 0$. Then $|X_ig|=\lambda$ for some $i\in I$. Consider the following two possible cases:

\vspace{0.7mm}
\noindent Case 1: Suppose $\lambda<\lambda_1$. By inequality \rm(i), we obtain $|X_{i\alpha^{-1}}f|\leq |X_{(i\alpha^{-1})\alpha}g|=|X_ig|=\lambda$. This yields $|X_{i\alpha^{-1}}f|<\lambda_1$, a contradiction of our hypothesis.

\vspace{0.7mm}
\noindent Case 2: Suppose $\lambda>\lambda_2$. By inequality \rm(ii), we obtain
$\lambda=|X_ig|\leq |X_{i\beta}f|$. This yields $|X_{i\beta}f|>\lambda_2$, a contradiction of our hypothesis.

\vspace{1mm}
In either case, we get a contradiction. Hence $n_{\lambda}^g=0$ for all cardinals $\lambda\notin \{\lambda_1,\lambda_2\}$, and thus $n_{\lambda}^f=0=n_{\lambda}^g$ for all cardinals $\lambda\notin \{\lambda_1,\lambda_2\}$.

\vspace{1mm}
Next, we prove that $n_{\lambda_1}^f=n_{\lambda_1}^g$. Consider the following two possible cases:

\vspace{0.5mm}
\noindent Case 1: Suppose $I_{\lambda_1}^f=\varnothing$. Suppose to the contrary that $I_{\lambda_1}^g\neq \varnothing$, and let $i\in I_{\lambda_1}^g$. Then $|X_ig|=\lambda_1$, and so $|X_{i\alpha^{-1}}f|\leq |X_{(i\alpha^{-1})\alpha}g|=|X_ig|=\lambda_1$ by inequality \rm(i). This yields $|X_{i\alpha^{-1}}f|=\lambda_1$ by hypothesis, and so $i\alpha^{-1}\in I_{\lambda_1}^f$. This is a contradiction, because $I_{\lambda_1}^f=\varnothing$. Hence $I_{\lambda_1}^g=\varnothing$, and thus $n_{\lambda_1}^f=n_{\lambda_1}^g$.

\vspace{1mm}
\noindent Case 2: Suppose $I_{\lambda_1}^f\neq \varnothing$. Let $i\in I_{\lambda_1}^f$. Then $|X_if|=\lambda_1$, and so $|X_{i\beta^{-1}}g|\leq |X_{(i\beta^{-1})\beta}f|=|X_if|=\lambda_1$ by inequality \rm(ii). Since $n_{\lambda}^g=0$ for all cardinals $\lambda\notin \{\lambda_1,\lambda_2\}$, it follows that $|X_{i\beta^{-1}}g|=\lambda_1$ whence $i\beta^{-1}\in I_{\lambda_1}^g$. Since $\beta^{-1}$ is injective, we thus get $n_{\lambda_1}^f\leq n_{\lambda_1}^g$.
Similarly, we can show that $i\alpha^{-1}\in I_{\lambda_1}^f$ for every $i\in I_{\lambda_1}^g$, and further prove that $n_{\lambda_1}^g\leq n_{\lambda_1}^f$ by using the fact that $\alpha^{-1}$ is injective.  Thus $n_{\lambda_1}^f=n_{\lambda_1}^g$.

\vspace{0.5mm}
In either case, we have $n_{\lambda_1}^f=n_{\lambda_1}^g$.

\vspace{0.5mm}
Next, we prove that $n_{\lambda_2}^f=n_{\lambda_2}^g$. Consider the following two possible cases:

\vspace{0.5mm}
\noindent Case 1: Suppose $I_{\lambda_2}^f=\varnothing$. Suppose to the contrary that $I_{\lambda_2}^g\neq \varnothing$, and let $i\in I_{\lambda_2}^g$. Then $|X_ig|=\lambda_2$, and so $\lambda_2=|X_ig|\leq |X_{i\beta}f|$ by inequality \rm(ii). This yields $|X_{i\beta}f|=\lambda_1$ by hypothesis, and so $i\beta\in I_{\lambda_2}^f$. This is a contradiction, because $I_{\lambda_2}^f=\varnothing$. Hence $I_{\lambda_2}^g=\varnothing$, and thus $n_{\lambda_2}^f=n_{\lambda_2}^g$.

\vspace{0.5mm}
\noindent Case 2: Suppose $I_{\lambda_2}^f\neq \varnothing$. Let $i\in I_{\lambda_2}^f$. Then $|X_if|=\lambda_2$, and so $\lambda_2=|X_if|\leq |X_{i\alpha}g|$ by inequality \rm(i). Since $n_{\lambda}^g=0$ for all cardinals $\lambda\notin \{\lambda_1,\lambda_2\}$, it follows that $|X_{i\alpha}g|=\lambda_2$ whence $i\alpha\in I_{\lambda_2}^g$. Since $\alpha$ is injective, we thus get $n_{\lambda_2}^f\leq n_{\lambda_2}^g$. Similarly, we can show that $i\beta\in I_{\lambda_2}^f$ for every $i\in I_{\lambda_2}^g$, and further prove that $n_{\lambda_2}^g\leq n_{\lambda_2}^f$ by using the fact that $\beta$ is injective.  Thus $n_{\lambda_2}^f=n_{\lambda_2}^g$.

\vspace{0.5mm}
In either case, we have $n_{\lambda_2}^f=n_{\lambda_2}^g$. Thus, since $n_{\lambda}^f=n_{\lambda}^g$ for all cardinals $\lambda$, we conclude from Proposition \ref{d-green-alter} that $(f,g)\in \mathcal{D}$ in $\mathcal{B}(X,\mathcal{P})$. Hence $g\in D_f$ as required.
\end{proof}

As an immediate consequence of Theorem \ref{two-consecutative}, we get

\begin{corollary}\label{lambda-1-2}
Let $f\in \mathcal{B}(X,\mathcal{P})$. If $n_{\lambda}^f=0$ for all cardinals $\lambda\geq 3$, then $D_f=J_f$ in $\mathcal{B}(X,\mathcal{P})$.	
\end{corollary}

In the following theorem, we give a necessary and sufficient condition for $\mathcal{D}=\mathcal{J}$ on $\mathcal{B}(X,\mathcal{P})$.

\begin{theorem}\label{D=J-BXP-size-3-finite}
Let $\mathcal{P} = \{X_i \colon i\in I\}$ be a partition of $X$. Then $\mathcal{D}=\mathcal{J}$ on $\mathcal{B}(X,\mathcal{P})$ if and only if $J=\{i\in I \colon |X_i|\geq 3\}$ is finite.
\end{theorem}

\begin{proof}[\textbf{Proof}]
Suppose that $\mathcal{D}=\mathcal{J}$ on $\mathcal{B}(X,\mathcal{P})$. If $\mathcal{P}$ is finite, then there is nothing to prove. Assume that $\mathcal{P}$ is infinite. Suppose to the contrary that $J=\{i\in I \colon |X_i|\geq 3\}$ is infinite. Then $\mathcal{D}\neq \mathcal{J}$ on $\mathcal{B}(X,\mathcal{P})$ by Proposition \ref{p3-infinite}, a contradiction. Hence $J=\{i\in I \colon |X_i|\geq 3\}$ is finite.

\vspace{0.1cm}	
Conversely, suppose that $J=\{i\in I \colon |X_i|\geq 3\}$ is finite. If $\mathcal{P}$ is finite, then $\mathcal{D}=\mathcal{J}$ on $\mathcal{B}(X,\mathcal{P})$ by Proposition \ref{p-finite}. Assume that $\mathcal{P}$ is infinite. In general, we have $\mathcal{D}\subseteq \mathcal{J}$. For the reverse inclusion, let $(f,g)\in \mathcal{J}$ in $\mathcal{B}(X,\mathcal{P})$. Then by Theorem \ref{b-j-green}, there exist $\alpha,\beta\in \Sym(I)$ such that for all $i\in I$,
\begin{eqnarray*}
|X_if|\leq |X_{i\alpha}g|\hspace{1cm} \cdots \text{(i)} \qquad\text{and}\qquad |X_ig|\leq |X_{i\beta}f| \hspace{1cm} \cdots \text{(ii)}
\end{eqnarray*}

In view of Proposition \ref{d-green-alter}, in order to show that $(f,g)\in \mathcal{D}$ in $\mathcal{B}(X,\mathcal{P})$, it suffices to prove that $n_{\lambda}^f=n_{\lambda}^g$ for all cardinals $\lambda$. Consider the following two possible cases:

\vspace{1mm}
\noindent Case 1: Suppose $n_{\lambda}^f=0$ for all $\lambda\geq 3$. Then $(f,g)\in \mathcal{D}$ by Corollary \ref{lambda-1-2}.

\vspace{1mm}
\noindent Case 2: Suppose $n_{\lambda}^f\neq 0$ for some cardinal $\lambda\geq 3$. Let $i\in I_{\lambda}^f$. Then $|X_if|=\lambda$, and so $3\leq |X_{i\alpha}g|$ by inequality \rm(i). Therefore $n_{\lambda}^g\neq 0$ for some cardinal $\lambda\geq 3$.

\vspace{0.5mm}
Let $\lambda_1,\ldots, \lambda_k$, where $\lambda_1>\cdots >\lambda_k$, be the cardinalities of blocks in $\mathcal{P}_f$ of cardinalities at least three; let $\mu_1,\ldots, \mu_t$, where $\mu_1>\cdots >\mu_t$, be the cardinalities of blocks in $\mathcal{P}_g$ of cardinalities at least three. Note by hypothesis that both $k$ and $t$ are finite.

\vspace{0.5mm}	
First, we prove that $\lambda_1=\mu_1$ and subsequently $n_{\lambda_1}^f=n_{\mu_1}^g$. For this, let $i\in I_{\lambda_1}^f$. Then by inequality \rm(i), we have $\lambda_1=|X_if|\leq |X_{i\alpha}g|\leq \mu_1$. Now, let $i\in I_{\mu_1}^g$. Then by inequality \rm(ii), we have $\mu_1=|X_ig|\leq |X_{i\beta}f|\leq \lambda_1$. Thus $\lambda_1=\mu_1$. To prove $n_{\lambda_1}^f=n_{\mu_1}^g$, we first observe for every $i\in I_{\lambda_1}^f$ that $|X_{i\alpha}g|=\lambda_1$, whence $i\alpha\in I_{\lambda_1}^g$. Therefore, since $\alpha$ is injective, we get $n_{\lambda_1}^f\leq n_{\lambda_1}^g$. Similarly, by using the facts that
$\beta$ is injective and $i\beta\in I_{\lambda_1}^f$ for every $i\in I_{\lambda_1}^g$, we get $n_{\lambda_1}^g\leq n_{\lambda_1}^f$. Thus $n_{\lambda_1}^f=n_{\lambda_1}^g$. Moreover, recall by hypothesis that both $I_{\lambda_1}^f$ and $I_{\lambda_1}^g$ are finite. Therefore, since $\alpha, \beta$ are injective, both maps $\alpha_{\lambda_1}\colon I_{\lambda_1}^f\to I_{\lambda_1}^g$ and $\beta_{\lambda_1}\colon I_{\lambda_1}^g\to I_{\lambda_1}^f$ defined by $i\alpha_{\lambda_1}=i\alpha$ and $i\beta_{\lambda_1}=i\beta$, respectively, are bijective.

\vspace{0.5mm}
Next, we prove that $\lambda_2=\mu_2$ and subsequently $n_{\lambda_2}^f=n_{\mu_2}^g$. For this, let $i\in I_{\lambda_2}^f$. Recall that $\alpha_{\lambda_1}\colon I_{\lambda_1}^f\to I_{\lambda_1}^g$ is bijective, and so $i\alpha\notin I_{\lambda_1}^g$. Therefore $|X_{i\alpha}g|\leq \mu_2$, which gives $\lambda_2=|X_if|\leq |X_{i\alpha}g|\leq \mu_2$ by inequality \rm(i). Now, let $i\in I_{\mu_2}^g$. Recall that $\beta_{\lambda_1}\colon I_{\lambda_1}^g\to I_{\lambda_1}^f$ is bijective, and so $i\beta\notin I_{\lambda_1}^f$. Therefore $|X_{i\alpha}f|\leq \lambda_2$, which gives $\mu_2=|X_ig|\leq |X_{i\beta}f|\leq \lambda_2$ by inequality \rm(ii). Thus $\lambda_2=\mu_2$.  To prove $n_{\lambda_2}^f=n_{\mu_2}^g$, we first observe for every $i\in I_{\lambda_2}^f$ that $|X_{i\alpha}g|=\lambda_2$, whence $i\alpha\in I_{\lambda_2}^g$. Therefore, since $\alpha$ is injective, we get $n_{\lambda_2}^f\leq n_{\lambda_2}^g$. Similarly, by using the facts that $\beta$ is injective and $i\beta\in I_{\lambda_2}^f$ for every $i\in I_{\lambda_2}^g$, we get $n_{\lambda_2}^g\leq n_{\lambda_2}^f$. Thus $n_{\lambda_2}^f=n_{\lambda_2}^g$. Moreover, recall by hypothesis that both $I_{\lambda_2}^f$ and $I_{\lambda_2}^g$ are finite. Therefore, since $\alpha, \beta$ are injective, both maps $\alpha_{\lambda_2}\colon I_{\lambda_2}^f\to I_{\lambda_2}^g$ and $\beta_{\lambda_2}\colon I_{\lambda_2}^g\to I_{\lambda_2}^f$ defined by $i\alpha_{\lambda_2}=i\alpha$ and $i\beta_{\lambda_2}=i\beta$, respectively, are bijective.

\vspace{0.5mm}
We can prove in similar way that $\lambda_j=\mu_j$ and $n_{\lambda_j}^f=n_{\mu_j}^g$ for all $j=3,\ldots,\min\{k,t\}$. Moreover, we observe for every $j=3,\ldots ,\min\{k,t\}$ that both maps $\alpha_{\lambda_j}\colon I_{\lambda_j}^f\to I_{\lambda_j}^g$ and $\beta_{\lambda_j}\colon I_{\lambda_j}^g\to I_{\lambda_j}^f$ defined by $i\alpha_{\lambda_j}=i\alpha$ and $i\beta_{\lambda_j}=i\beta$, respectively, are bijective.

\vspace{1.0mm}
We now claim that $k = t$. Suppose to the contrary that $k\neq t$. Assume without loss of generality that $k<t$. Then $\min\{k,t\}=k$, and so $\lambda_j=\mu_j$ for all $j \in [k]$. Now, let $i\in I_{\mu_t}^g$.
Then we see that $i\beta\notin I_{\lambda_j}^f$ for all $j \in [k]$, since  $\beta_{\lambda_j}\colon I_{\lambda_j}^g\to I_{\lambda_j}^f$ is bijective for all $j \in [k]$. Therefore $i\beta\in I_1^f\cup I_2^f$, and so $|X_{i\beta}f|\le 2$. Recall that $|X_ig|= \mu_t\geq 3$. By inequality \rm(ii), we therefore obtain $3\leq |X_ig|\leq |X_{i\beta}f|\le 2$, a contradiction. Hence $k=t$.

\vspace{0.5mm}
Next, we prove that $n_2^f=n_2^g$. Consider the following two possible cases:

\vspace{0.5mm}
\noindent Case I: Suppose $I_2^f=\varnothing$. Suppose to the contrary that $I_2^g\neq \varnothing$, and let $i\in I_2^g$. Then $|X_ig|=2$. Notice that $i\beta\notin I_{\lambda_j}^f$ for all $j \in [k]$, since $\beta_{\lambda_j}\colon I_{\lambda_j}^g\to I_{\lambda_j}^f$ is bijective for all $j \in [k]$. Therefore $i\beta\in I_1^f$, and so $|X_{i\beta}f|=1$. By inequality \rm(ii), we obtain $2=|X_ig|\leq |X_{i\beta}f|=1$, a contradiction. Hence $I_2^g= \varnothing$, and thus $n_2^f=n_2^g$.

\vspace{0.7mm}
\noindent Case II: Suppose $I_2^f\neq \varnothing$. Let $i\in I_2^f$. Then $|X_if|=2$, and so $2\leq |X_{i\alpha}g|$ by inequality \rm(i). Notice that $i\alpha\notin I_{\lambda_j}^g$ for all $j \in [k]$, since $\alpha_{\lambda_j} \colon I_{\lambda_j}^f\to I_{\lambda_j}^g$ is bijective for all $j \in [k]$. Therefore $i\alpha\in I_2^g$. Since $\alpha$ is injective, we thus get $n_2^f\leq n_2^g$. Now, let $i\in I_2^g$. Then $|X_ig|=2$, and so $2\leq |X_{i\beta}f|$  by inequality \rm(ii). Notice that $i\beta \notin I_{\lambda_j}^f$ for all $j \in [k]$, since $\beta_{\lambda_j} \colon I_{\lambda_j}^g\to I_{\lambda_j}^f$ is bijective for all $j \in [k]$. Therefore $i\beta\in I_2^f$. Since $\beta$ is injective, we thus get $n_2^f\geq n_2^g$. Hence $n_2^f=n_2^g$.

\vspace{1mm}	
In either case, we have $n_2^f=n_2^g$. Finally, we prove that $n_1^f=n_1^g$. Consider the following two possible cases:

\vspace{0.7mm}
\noindent Case I: Suppose $I_1^f= \varnothing$. Suppose to the contrary that $I_1^g\neq \varnothing$, and let $i\in I_1^g$. Then $|X_ig|=1$, and so $|X_{i\alpha^{-1}}f|\leq |X_{(i\alpha^{-1})\alpha}g|=|X_ig|$ by inequality \rm(i). It follows that $|X_{i\alpha^{-1}}f|=1$, and so $i\alpha^{-1}\in I_1^f$, a contradiction. Hence $I_1^g= \varnothing$, and thus $n_1^f=n_1^g$.

\vspace{0.7mm}
\noindent Case II: Suppose $I_1^f\neq \varnothing$. Let $i\in I_1^f$. Then $|X_if|=1$, and so $|X_{i\beta^{-1}}g|\leq |X_{(i\beta^{-1})\beta}f|=|X_if|$ by inequality \rm(ii). It follows that $|X_{i\beta^{-1}}g|=1$, and so $i\beta^{-1}\in I_1^g$. Since $\beta^{-1}$ is injective, we thus get $n_1^f\leq n_1^g$. Now, let $i\in I_1^g$. Then $|X_ig|=1$, and so $|X_{i\alpha^{-1}}f|\leq |X_{(i\alpha^{-1})\alpha}g|=|X_ig|$ by inequality \rm(i). It follows that $|X_{i\alpha^{-1}}f|=1$, and so $i\alpha^{-1}\in I_1^f$. Since $\alpha^{-1}$ is injective, we thus get $n_1^f\geq n_1^g$. Hence $n_1^f=n_1^g$.

\vspace{1mm}	
In either case, we have $n_1^f=n_1^g$. In addition, we note for every cardinal $\lambda \notin \{1,2,\lambda_k,\ldots,\lambda_1\}$ that $n_{\lambda}^f=0=n_{\lambda}^g$. Thus we conclude from Proposition \ref{d-green-alter} that $(f,g)\in \mathcal{D}$ in $\mathcal{B}(X,\mathcal{P})$ as required.
\end{proof}

Note that $\mathcal{D}=\mathcal{J}$ on $\mathcal{B}(X,\mathcal{P})$ if $\mathcal{P}$ is a trivial partition of $X$. In connection with Proposition \ref{p-finite}, Theorem \ref{two-consecutative}, and Theorem \ref{D=J-BXP-size-3-finite}, we end this section with the following conjecture.

\vspace{2mm}
\noindent\textbf{Conjecture.} Let $\mathcal{P}=\{X_i\colon i\in I\}$ be a partition of $X$ and $f\in \mathcal{B}(X,\mathcal{P})$. Then  $D_f=J_f$ in $\mathcal{B}(X,\mathcal{P})$ if and only if there exist any two consecutive cardinals $\lambda_1,\lambda_2$ and a finite subset $K$ of $I$ such that $|X_if|\in \{\lambda_1,\lambda_2\}$ for all $i\in I\setminus K$.


\end{document}